\documentclass[11pt]{amsart}
\usepackage{amssymb}
\usepackage{palatino}
\usepackage{tikz-cd}
\input amssym.def

\usepackage{amsmath, amsfonts}
\usepackage{amssymb}
\usepackage{amscd}
\usepackage[mathscr]{eucal}
\usepackage{palatino}
\usepackage{comment}
\usepackage{enumerate}

\newfont{\cyrr}{wncyr10}

\newcommand{\thmref}[1]{Theorem~\ref{#1}}

\newcommand{\propref}[1]{Proposition~\ref{#1}}
\newcommand{\lemref}[1]{Lemma~\ref{#1}}

\newtheorem{thm}{Theorem}

\newtheorem{lem}[thm]{Lemma}

\newtheorem{prop}[thm]{Proposition}
\newtheorem{rmk}{Remark}[section]

\def\({\left(}
\def\){\right)}
\def\[{\left[}
\def\]{\right]}

\def\N{\mathbb{N}}

\def\cA{\mathcal{A}}

\def\cC{\mathcal{C}}

\def\cS{\mathcal{S}}

\newcommand{\e}{\epsilon}

\newcommand{\lf}{\lfloor}
\newcommand{\rf}{\rfloor}

\title{On coprimality of consecutive elements in certain sequences}

\author{Jean-Marc Deshouillers and Sunil Naik}

\address{Jean-Marc Deshouillers \newline
Institut de Math\'ematiques de Bordeaux,
Universit\'e de Bordeaux, CNRS, Bordeaux INP
33400, Talence, France}
\email{jean-marc.deshouillers@math.u-bordeaux.fr}

\address{Sunil L Naik \newline
Department of Mathematics,
Queen's University, Jeffrey Hall, 
99 University Avenue, 
Kingston, ON K7L3N6, 
Canada}
\email{naik.s@queensu.ca}

\begin{document}
	
	\hfuzz 5pt
	
\subjclass[2020]{11B05, 11B50, 11K31, 11N56}
	
\keywords{Segal-Piatetski-Shapiro sequences, Regular sequences, Pairwise coprime, Banach density}
	
\maketitle

\begin{flushright}
\textit{Dedicated to Kalyan Chakraborty and Srinivas Kotyada \\on the occasion of their 60th anniversaries}
\end{flushright}
	
\begin{abstract}
The study of finding blocks of primes 
in certain arithmetic sequences is
one of the classical problems in number theory. 
It is also very interesting to study blocks of 
consecutive elements in such sequences that are pairwise coprime. 
In this context, 
we show that if $f$ is a twice continuously differentiable 
real-valued function on $[1, \infty)$ such that 
$f''(x) \to 0$ as $x \to \infty$ and 
$\limsup_{x \to \infty} f'(x) = \infty$, 
then there exist arbitrarily long blocks of 
pairwise coprime consecutive elements 
in the sequence $(\lfloor f(n) \rfloor)_n$. 
This result refines the qualitative part 
of a recent result by the first author, Drmota and M\"{u}llner.

We also prove that there exists a subset $\cA \subseteq \mathbb{N}$ 
having upper Banach density one such that for any two distinct integers $m, n \in \cA$, 
the integers $\lfloor f(m) \rfloor$ and $\lfloor f(n) \rfloor$ are pairwise coprime. 
Further, we show that there exist arbitrarily long blocks of consecutive elements 
in the sequence $(\lfloor f(n) \rfloor)_n$ such that no two of them are pairwise coprime. 
\end{abstract}

\section{Introduction and Statements of Results}
The Segal-Piatetski-Shapiro sequences are sequences 
of the form $(\lfloor n^c \rfloor)_n$ for a fixed $c \in \(1, \infty\) \backslash \N$. 
In \cite{PS}, Piatetski-Shapiro proved that there exist 
infinitely many primes in  the sequence $(\lfloor n^c \rfloor)_n$ if $c \in\(1, \frac{12}{11}\)$. 
In this article, we are interested in finding blocks of consecutive elements 
in $(\lfloor n^c \rfloor)_n$ that are pairwise coprime 
and of course, it is nice to have all of them distinct primes. 
In a recent work \cite{DDM}, the first author, Drmota and M\"{u}llner
showed that if $c \in (1, 2)$ and $0 < \alpha < \min\(c - 1, 1 - \frac{c}{2}\)$, 
then there exist infinitely many positive integers $n$ such that for any positive integer 
$H \leq \alpha \log n$, all the elements in the sequence 
$\{ \lf n^c \rf,  \lf (n+1)^c \rf, \cdots, \lf (n+H)^c \rf \}$ 
are pairwise coprime. We refer the reader to \cite{BS,BR,DD,EL,LM,Wat} for related literature. 

In this article, we extend the qualitative part of the above result to a larger family of regular sequences.
\begin{thm}\label{copf}
Let $f$ be a twice continuously differentiable 
real valued function on $[1, \infty)$ such that
\begin{enumerate}[(i)]
\item  $f''(x) \to 0$ as $x \to \infty$,
\item  $\limsup_{x \to \infty} f'(x) = \infty$.
\end{enumerate}
Then for any positive integer $H$, 
there exist infinitely many positive integers $n$ such that 
all the elements in the sequence 
$\{ \lf f(n) \rf,  \lf f(n+1) \rf, \cdots, \lf f(n+H) \rf \}$ 
are pairwise coprime. 	
\end{thm}

\begin{rmk}
For $c\in (1, 2)$, the function $f(x)=x^c$ 
satisfies the conditions of \thmref{copf}. 
This implies the existence of arbitrarily long blocks of 
consecutive elements in the sequence $\(\lfloor n^c \rfloor\)_n$ 
which are pairwise coprime, thus recovering the qualitative part 
of the main result of \cite{DDM}.
\end{rmk}

\begin{rmk} 
We note that the second condition in \thmref{copf} is necessary for the existence 
of  arbitrarily long blocks of consecutive  pairwise coprime elements 
in the sequence $\(\lfloor f(n) \rfloor\)_n$. Clearly, the case when 
$\limsup_{x \to \infty} f'(x) = -\infty$ can be treated 
similarly by considering the function $-f$.  Further, 
one can consider the case when the second condition in \thmref{copf}
is replaced with $\liminf_{x \to \infty} f'(x) = -\infty$ in a similar fashion. 
Without loss of generality, we may assume that $\liminf_{x \to \infty} f'(x) > -\infty$.
Suppose  
$\limsup_{x \to \infty} f'(x)$ is finite, then $| f(m+1)-f(m)| = O(1) $ 
and hence for any $n \in \N$, there exists an interval $I$ of length $O(H)$ 
containing$\{ \lf f(n) \rf,  \lf f(n+1) \rf, \cdots, \lf f(n+H) \rf \}$. 
If $\lf f(n) \rf$,  $\lf f(n+1) \rf$, $\cdots, \lf f(n+H) \rf$ are pairwise coprime, 
then $I$ contains a set of cardinality at least $H$ 
whose elements are pairwise coprime. But this is impossible by a 
result of Erd\H{o}s and Selfridge  \cite[p. 5]{ES} for sufficiently large $H$. 
This shows that  $\limsup_{x \to \infty} f(x) = \infty$ is necessary in \thmref{copf}.
\end{rmk}

\begin{rmk} We also note that the first condition in \thmref{copf} cannot be replaced 
by the weaker condition $f'(x) = o(x)$. For $n\geq 2$, 
let $a_n$ be an even integer in $[n^{3/2}-2 , n^{3/2}]$. 
Observe that $a_n < a_{n+1} $ and $a_{n+1} - a_n = \(3/2 + o(1)\) n^{1/2}$. 
One can construct a twice  continuously differentiable 
real valued function $g$ on $[1, \infty)$ such that $g(n) = a_n$  
and $x^{1/2} < g'(x) < 2x^{1/2}$ for sufficiently large $x$. 
Then we have $g'(x) = o(x)$ and $\limsup_{x \to \infty} g'(x) = \infty$, 
but $g(n)$ is even for all $n \geq 2$. 

We can also consider
$$
g(x) ~=~ \int_{1}^{x} t^{\frac{1}{2}} \sin(2\pi t) dt.
$$
Then $g'(x) = o(x)$ and $\limsup_{x \to \infty} g'(x) = \infty$. 
Note that
$$
g(n+1) - g(n) ~=~ \int_{n}^{n+1} t^{\frac{1}{2}} \sin(2\pi t) dt 
~=~ \frac{1}{2}\int_{0}^{1} \(\(\frac{y}{2}+n\)^{\frac{1}{2}} 
- \(\frac{y+1}{2}+n\)^{\frac{1}{2}}\) \sin(\pi y) dy.
$$
Thus 
$$
|g(n+1) ~-~ g(n)| ~<~ \frac{1}{\sqrt{n}}.
$$
Hence for all sufficiently large $n$, any three consecutive elements in the sequence 
$\(\lf g(n) \rf\)_n$ are not pairwise coprime.
\end{rmk}

In order to state the next result, let us recall the following notion of density. 
For a subset $\cS$ of natural numbers, the upper Banach density (or upper uniform density,
see \cite{GTT, Ri} for more details) of $\cS$ is defined by
$$
\lim_{H \to \infty} \limsup_{x \to \infty} \frac{\#\( \cS \cap (x, x+H]\)}{H}.
$$

In this set-up, we prove the following result.
\begin{thm}\label{copBnc1}
Let $f$ be as in \thmref{copf}. Then there exists a subset $\cA$ of 
natural numbers with upper Banach density equal to $1$ such that for any 
distinct pair of integers $m$ and $n$ contained in $\cA$, 
the integers $\lf f(m) \rf$ and $\lf f(n) \rf$ are pairwise coprime.
\end{thm}

In the opposite direction to \thmref{copf}, it is also possible to find arbitrarily long blocks of elements 
in the sequence $\(\lf f(n) \rf\)_n$ 
such that no two elements are pairwise coprime. In fact, we prove the following theorem.
\begin{thm}\label{conevenf}
Let $f$  be as in \thmref{copf}. 
Then for any positive integer $H$, there exist infinitely many positive integers $n$ 
such that all the elements in the sequence 
$\{ \lf f(n) \rf,  \lf f(n+1) \rf, \cdots, \lf f(n+H) \rf \}$ are even.
\end{thm}

\medspace

\section{Proofs of the theorems}
The following proposition is at the heart of our proof of \thmref{copf}.
\begin{prop}\label{prcopf}
Let $H\geq 2$ be a positive integer and $\Pi_H = \prod_{p \leq H} p$ be the product 
of all primes less than or equal to $H$. Let $f \in \cC^2\([1, \infty)\)$ be a twice continuously differentiable 
real valued function on $[1, \infty)$ and $n$ be a positive integer satisfying
\begin{eqnarray}\label{H1}
&\{f(n)\} ~\leq~ \frac{1}{3},\phantom{m}  
\frac{1}{9H} ~\leq~ \{f'(n)\}  ~\leq~ \frac{1}{3H}  
\phantom{m}\text{and}\phantom{m}
|f''(x)| ~\leq~ \frac{1}{10H^2} ~~\text{ for }~~ x \in [n, n+H],\\ \label{H2}
&\lf f'(n) \rf ~\equiv~ 0 ~\(\text{ mod } \Pi_H\), \\ \label{H3}
&\gcd\(\lf f(n) \rf, ~ \lf f'(n) \rf\) ~=~ 1.
\end{eqnarray}
Then the integers in the set $\{f(n+h) : h \in [H/2, H] \cap \N\}$ are pairwise coprime.
\end{prop}

\begin{proof}
By Taylor's theorem, for any integer $h \in [H/2, H]$, there exists $\theta_h \in [0,1]$ such that
\begin{equation*}
\begin{split}
f(n+h) 
&~=~ f(n) ~+~ h f'(n) ~+~ \frac{h^2}{2} f''(n+\theta_h h)\\
&~=~ \lf f(n) \rf ~+~ h \lf f'(n) \rf ~+~ \{f(n)\} ~+~ h\{f'(n)\} 
~+~ \frac{h^2}{2} f''(n+\theta_h h).
\end{split}\end{equation*}
Note that
$$
0 ~\leq~ 
f(n+h) ~-~ \lf f(n) \rf ~-~ h \lf f'(n)\rf 
~\leq~
\frac{1}{3} ~+~ \frac{1}{3} ~+~ \frac{1}{20} 
~<~ 1.
$$
This implies that for any integer $h \in [H/2,H]$, we have
\begin{equation}\label{fn+h}
\lf f(n+h) \rf  ~=~ \lf f(n) \rf ~+~ h \lf f'(n) \rf.
\end{equation}
Observe that for any prime $p \leq H$, we have $ p \nmid \lf f(n+h) \rf$ 
for any $H/2 \leq h \leq H$. 
This is because, from \eqref{H2}, we have $p\mid \lf f'(n) \rf$ and 
from \eqref{H3}, $ p \nmid \lf f(n) \rf$ 
and thus we get $p\nmid \lf f(n+h) \rf$ by \eqref{fn+h}.
Suppose that there exists a prime $p>H$ 
such that $p \mid \gcd\(\lf f(n+h) \rf, \lf f(n+k) \rf\)$ 
for some pair of integers $h$ and $k$ with $H/2 \leq k < h \leq H$. 
Then $p\mid \lf f(n+h) \rf - \lf f(n+k) \rf$. From \eqref{fn+h}, we have
$$
\lf f(n+h) \rf ~-~ \lf f(n+k) \rf ~=~ (h-k) \lf f'(n) \rf.
$$
Hence $p \mid \lf f'(n) \rf$. 
From \eqref{fn+h}, we deduce that $p \mid \lf f(n)\rf$, 
which is a  contradiction to \eqref{H3}. Thus $\lf f(n+h) \rf$ and 
$\lf f(n+k)\rf$ are coprime for any two distinct integers $h, k \in  [H/2, H]$.
\end{proof}

\medspace

We need the following ``folklore" lemma which gives an equivalence criterion 
for $\lf x \rf$ to be even.
\begin{lem}\label{fracx02}
	Let $x$ be a real number. We have
	$$
	\lf x \rf ~\equiv~ 0 ~\( \text{ mod } 2\) 
	~\iff~
	\left\{\frac{x}{2}\right\} ~\in~ \left[0,~ \frac{1}{2}\right).
	$$
\end{lem}

\medspace

\subsection{Proof of \thmref{copf}}
Let $f$ be as in \thmref{copf} and $H \geq 2$ be a large positive integer. 
Also let $x_0 > 1$ be a real number such that
\begin{equation}\label{8}
|f''(x)| ~\leq~ \frac{1}{(100 H \Pi_H)^3}
\end{equation}
for $x \geq x_0$.
Let $q > H$ be a prime number which is sufficiently large. 
Then there exists a positive integer $m = m(q) > x_0 + 1$ such that
\begin{equation}\label{9}
\lf f'(m-1) \rf ~<~ \lf f'(m) \rf 
~=~ q \Pi_H.
\end{equation}
By the mean value theorem and \eqref{8},
we have for any $x \geq x_0$,
\begin{equation}\label{10}
| f'(x + 1) ~-~ f'(x) | ~\leq~ \frac{1}{(100 H \Pi_H)^3}
\end{equation}
which implies in conjunction with the inequality in \eqref{9} that
\begin{equation}\label{11}
\{f'(m)\} ~<~ \frac{1}{100H}.
\end{equation}
From \eqref{10}, \eqref{11} and the fact that $\limsup_{x \to \infty}f'(x) = \infty$, 
there exists an integer $n_0 \geq m$ such that
\begin{equation}\label{12}
\lf f'(n_0) \rf = \lf f'(m) \rf = q \Pi_H
\phantom{m}\text{and}\phantom{m} 
\{f'(n_0)\} ~\in~ \(\frac{1}{6H}, ~ \frac{1}{5H} \).
\end{equation}
Let $K = 15 H \Pi_H$. 
By Taylor's formula, for any integer $k \in [0, K]$, we have
$$
f(n_0 + k) ~=~ f(n_0) ~+~ k f'(n_0) ~+~ \frac{k^2}{2} f''(n_0 +\theta_k k)
$$
for some $\theta_k \in [0,1]$. Let $\e_k = \frac{k^2}{2} f''(n_0 +\theta_k k)$.
From \eqref{8}, we have
\begin{equation}\label{13}
	|\e_k| ~<~ \frac{1}{20H}.
\end{equation}
We can write
\begin{equation*}
	f(n_0 + k) 
	~=~ f(n_0) ~+~ k q \Pi_H ~+~ k \{f'(n_0)\} ~+~ \e_k.
\end{equation*}
Let us consider the sequence $\(a_k\)_{0 \leq k \leq K}$, 
where 
$$
a_k ~=~f(n_0 + k) - k q \Pi_H ~=~ f(n_0) ~+~ k \{f'(n_0)\} ~+~ \e_k.
$$
From \eqref{12} and \eqref{13}, we have 
$$
a_K ~-~ a_0 ~=~ K \{f'(n_0)\} ~+~ \e_K 
~>~ 15H\Pi_H \cdot \frac{1}{6H} ~-~ \frac{1}{20H}
~>~ 2 \Pi_H.
$$
For any $k \in [0, K-1]$, we get
\begin{equation}\label{akdiff}
0 ~<~ a_{k+1} ~-~ a_k 
~=~ \{f'(n_0)\} ~+~ \e_{k+1} ~-~ \e_k 
~<~ \frac{1}{5H} ~+~ \frac{1}{20H} ~+~ \frac{1}{20H}
~<~ \frac{1}{3H}.
\end{equation}
Thus, there exists an integer $k_0 \in [0, K]$ such that 
$$
\lf a_{k_0} \rf ~\equiv~  1 ~\(\text{ mod } \Pi_H\), 
~~ \lf a_{k_0} \rf ~\not\equiv~ 0 ~\(\text{ mod } q\) 
~~\text{ and }~~
\{a_{k_0}\} ~<~ \frac{1}{3}.
$$
To see this, let $b \in [a_0, a_K)$ be the smallest integer 
such that
$$
b ~\equiv~  1 ~\(\text{ mod } \Pi_H\)
\phantom{m}\text{and}\phantom{m}
~~ b ~\not\equiv~ 0 ~\(\text{ mod } q\).
$$
Such an integer exists, since $a_K - a_0 > 2 \Pi_H$. 
Let $k_0 \in [0, K]$ be the smallest integer such that 
$a_{k_0} \geq b$. Then from \eqref{akdiff}, 
we have $\lf a_{k_0} \rf = b$ and $\{a_{k_0}\} < \frac{1}{3H}$. 
We set $n=n_0+k_0$. It is easy to see that
\begin{equation}\label{eqnqProp}
\begin{split}
&\lf f'(n) \rf ~=~ q \Pi_H, \phantom{m} 
\gcd \(\lf f(n) \rf,~ q \Pi_H\) ~=~1, \\
& \{f(n)\} ~\leq~ \frac{1}{3}, \phantom{m} 
\frac{1}{9H} ~\leq~ \{f'(n)\} ~\leq~ \frac{1}{3H}, \phantom{m}
|f''(n)| ~\leq~ \frac{1}{10H^2}.
\end{split}
\end{equation}
Hence there exist infinitely many pairs positive integers $q$, $n = n(q)$ satisfying \eqref{eqnqProp}.
Now \thmref{copf} follows from \propref{prcopf}.  \qed

\medspace

\subsection{Proof of \thmref{copBnc1}}
Let $f$ be as in \thmref{copf}. Let $H_1 \geq 2$ be a natural number. 
Then from \thmref{copf}, there exists $n_1 \in \N$ such that the integers
$ \lf f(n_1) \rf,  \lf f(n_1+1) \rf, \cdots, \lf f(n_1+H_1) \rf$ 
are pairwise coprime. Let $H_2$ be a natural number strictly greater than 
$H_1 + \max_{0 \leq h \leq H_1} |f(n_1+h)|$. Now proceeding as in 
the proof of \thmref{copf}, we can find a prime $q_2 > H_2$ 
and a natural number $n_2 > n_1+H_1$ such that
$$
\lf f'(n_2) \rf ~=~ q_2 \Pi_{H_2},~~ \gcd\(\lf f(n_2) \rf, q_2 \Pi_{H_2} \) ~=~ 1 
\phantom{m}\text{and}\phantom{m}
\lf f(n_2+h) \rf ~=~ \lf f(n_2) \rf + h \lf f'(n_2) \rf
$$
for $H_2/2 \leq h \leq H_2$. Then the integers in the set 
$\{ \lf f(n_2+h) \rf : h \in [H_2/2, H_2] \cap \N\}$ are pairwise coprime. 
Also, if $p$ is a prime which divides $\lf f(n_1+h) \rf $ for some $ 0\leq h \leq H_1$, 
then $p$ divides $\Pi_{H_2}$, since $H_2 > H_1 + \max_{0 \leq h \leq H_1} |f(n_1+h)|$. 
Hence $p$ does not divide $\lf f(n_2+h) \rf$ for  any integer  $h \in [H_2/2, H_2] $.  
By induction, there exist integers 
$H_r > H_{r-1} + \max\{|f(n_i+h_j)| : h_j \in [H_i/2, H_i] \cap \N, 1 \leq i \leq r-1 \}$ 
and $n_r > n_{r-1} + H_{r-1}$ such that the integers in the set 
$\{ \lf f(n_i+h_j) \rf : h_j \in [H_i/2, H_i] \cap \N,~ 1 \leq i \leq r\}$ are pairwise coprime. 
Let
$$
\cA ~=~ \{ n_i+h_j ~:~ h_j \in [H_i/2, H_i] \cap \N,~ i \in \N\}.
$$
Then clearly the upper Banach density of $\cA$ is equal to $1$. This completes the proof of \thmref{copBnc1}. \qed

\medspace

\subsection{Proof of \thmref{conevenf}}
Arguing as in the proof of \thmref{copf} 
(cf. \eqref{8} and \eqref{12}  for the function $f/2$), 
for given $H$, we can find infinitely many positive integers $n$ such that 
$$
\left\{\frac{f'(n)}{2}\right\} 
~\in~ \( \frac{1}{6H},~ \frac{1}{5H}\)
\phantom{m}\text{and}\phantom{m} 
|f''(x)| ~\leq~ \frac{1}{1000H^3} 
$$
for $x \geq n$.
The Taylor expansion of $f$ leads to 
$$
\frac{f(n+h)}{2} ~=~ 
\frac{f(n)}{2} ~+~ h \frac{f'(n)}{2} ~+~ \e_h
$$
with $ |\e_h| < \frac{1}{50H}$ for $0 \leq h \leq 10H$.
We have
$$
\left\{ \frac{f(n+h)}{2} \right\} ~=~
\left\{ \left\{ \frac{f(n)}{2}\right\} ~+~
h \left\{\frac{f'(n)}{2}\right\} 
~+~ \e_h \right\}.
$$
Set $\xi_h = h \left\{\frac{f'(n)}{2}\right\} + \e_h$. 
Then we have
\begin{equation*}
	\xi_{10H} ~>~ 1 
	\phantom{m}\text{and}\phantom{m} 
	0 ~<~ \xi_{h+1} ~-~ \xi_h ~<~ \frac{1}{4H}
\end{equation*}
for any integer $h \in [0, 10H)$.
Thus one can find at least $H$ consecutive values of $h \in [0, 10H)$ such that $\left\{\frac{f(n+h)}{2}\right\} \in \left[ 0,~ \frac{1}{2}\right)$, 
which, thanks to \lemref{fracx02}, proves \thmref{conevenf}. \qed

\medspace

\section{Acknowledgments}
This paper was supported by the joint FWF-ANR project Arithrand:  FWF: I 4945-N and ANR-20-CE91-0006 and by the SPARC project 445.
The second author would like to thank Queen's University, Canada, and the Institute of Mathematical Sciences (IMSc), India, for providing an excellent atmosphere for work. A part of the work was completed when the second author was visiting the Institut de Math\'ematiques de Bordeaux, France, and the second author would like to acknowledge the hospitality during the visit.

\medspace


\begin{thebibliography}{xx}
    \bibitem{BS}
	W. Banks and I. E. Shparlinski,
	{\em On the greatest common divisor of integer parts of polynomials}, 
	arXiv:2205.00253.
	
	\bibitem{BR}
	V. Bergelson and F. K. Richter, 
	{\em On the density of coprime tuples of the form 
	$(n, \lf f_1(n)\rf , \cdots, \lf f_k(n)\rf)$, where $f_1,\cdots, f_k$ are functions from a Hardy field}, Number theory - Diophantine problems, uniform distribution and applications, 109--135, Springer, Cham, 2017.
	
	\bibitem{DD}
	F. Delmer and J.-M. Deshouillers, 
	{\em On the probability that $n$ and $[n^c]$ are coprime}, 
	Period. Math. Hungar. \textbf{45} (2002), no. 1-2, 15--20.
	
	\bibitem{DDM}
	J.-M. Deshouillers, M. Drmota and C. M\"{u}llner,
	{\em Coprimality of consecutive elements in a Piatetski-Shapiro sequence}, 
	Number theory in memory of Eduard Wirsing, 91--98, Springer, Cham (2023).

	
	\bibitem{EL}
	P. Erd\H{o}s and G. G. Lorentz, 
	{\em On the probability that $n$ and $g(n)$ are relatively prime}, 
	Acta Arith. \textbf{5} (1958), 35--44 (1959).
	
	\bibitem{ES}
	P. Erd\H{o}s and J. L. Selfridge, 
	{\em Complete prime subsets of consecutive integers}, Proceedings of the Manitoba Conference on Numerical Mathematics (Univ. Manitoba, Winnipeg, Man., 1971), 1--14.
	
	\bibitem{GTT}
	G. Grekos, V. Toma and J. Tomanov\'a,
	{\em A note on uniform or Banach density}, 
	Ann. Math. Blaise Pascal {\bf 17} (2010), no. 1, 153--163.
	

	\bibitem{LM}
	J. Lambek and L. Moser, 
	{\em On integers $n$ relatively prime to $f(n)$}, 
	Canadian J. Math. \textbf{7} (1955), 155--158.
	
	\bibitem{PS}
	I. I. Piatetski-Shapiro,
	{\em On the distribution of prime numbers in sequences of the form $[f(n)]$}, (Russian), Mat. Sbornik N.S. \textbf{33}(75), (1953), 559--566.
	
	\bibitem{Ri}
	P. Ribenboim,
	{\em Density results on families of Diophantine equations with finitely many solutions}, Enseign. Math. (2) {\bf 39} (1993), no. 1--2, 3--23.
	
	
	\bibitem{Wat}
	G. L. Watson, 
	{\em On integers $n$ relatively prime to $[\alpha n]$}, 
	Canad. J. Math. \textbf{5} (1953), 451--455.
	
\end{thebibliography}
\end{document}